\documentclass[graybox]{svmult}

\usepackage{type1cm}        % activate if the above 3 fonts are
\usepackage{makeidx}         % allows index generation
\usepackage{graphicx}        % standard LaTeX graphics tool
\usepackage{multicol}        % used for the two-column index
\usepackage[bottom,stable]{footmisc}% places footnotes at page bottom

\usepackage{newtxtext}       % 
\usepackage[varvw]{newtxmath}       % selects Times Roman as basic font

\usepackage{graphicx}       % standard LaTeX graphics tool
\usepackage[latin1]{inputenc}
\usepackage[T1]{fontenc}
\usepackage{microtype} % good font tricks

\usepackage{cite}
\usepackage{array,colortbl}
\usepackage{amsmath,amsfonts,bm,mathtools,mathrsfs}
\DeclareFontFamily{U}{mathx}{\hyphenchar\font45}
\DeclareFontShape{U}{mathx}{m}{n}{
      <5> <6> <7> <8> <9> <10>
      <10.95> <12> <14.4> <17.28> <20.74> <24.88>
      mathx10
      }{}
\DeclareSymbolFont{mathx}{U}{mathx}{m}{n}
\DeclareFontSubstitution{U}{mathx}{m}{n}
\DeclareMathAccent{\widecheck}{0}{mathx}{"71}

\def\citep#1#2{\cite[{#1}]{#2}}

\newcommand{\RefSec}[1]{Section~\textup{\ref{#1}}}

\newcommand{\RefThm}[1]{Theorem~\textup{\ref{#1}}}
\newcommand{\RefProp}[1]{Proposition~\textup{\ref{#1}}}
\newcommand{\RefLem}[1]{Lemma~\textup{\ref{#1}}}

\newcommand{\RefCol}[1]{Corollary~\textup{\ref{#1}}}

\usepackage{algorithm}
\usepackage{algorithmic}
\newcommand{\N}{{\mathbb{N}}} % natural numbers {1, 2, ...}
\newcommand{\R}{{\mathbb{R}}} % reals
\newcommand{\bsi}{{\boldsymbol{i}}}
\newcommand{\bsj}{{\boldsymbol{j}}}

\newcommand{\bsu}{{\boldsymbol{u}}}
\newcommand{\bsv}{{\boldsymbol{v}}}
\newcommand{\bsw}{{\boldsymbol{w}}}
\newcommand{\bsx}{{\boldsymbol{x}}}
\newcommand{\bsy}{{\boldsymbol{y}}}

\newcommand{\bsgamma}{{\boldsymbol{\gamma}}}

\newcommand{\bslambda}{{\boldsymbol{\lambda}}}

\newcommand{\abs}[1]{\left\vert#1\right\vert}
\newcommand{\norm}[1]{\left\Vert#1\right\Vert}
\newcommand{\innProd}[3]{\left\langle #1, #2 \right\rangle_{#3}}

\newcommand{\Hu}{{H_{\bsu}}}
\newcommand{\Hv}{{H_{\bsv}}}
\newcommand{\Hg}{H_{\bsgamma}}
\newcommand{\HgA}{H_{\bsgamma}^{\textup{A}}}

\newcommand{\Ltwou}{L^2(\Du, \rhou)}
\newcommand{\Ltwo}{L^2(\X, \rho^\N)}

\newcommand{\drho}{\textup{d}\rho}

\newcommand{\decg}{\textup{d}_\bsgamma}

\newcommand{\decup}{\textup{d}^{\textup{up}}}
\newcommand{\declow}{\textup{d}^{\textup{low}}}
\newcommand{\decgup}{\textup{d}^{\textup{up}}_\bsgamma}
\newcommand{\decglow}{\textup{d}^{\textup{low}}_\bsgamma}
\newcommand{\declup}{\textup{d}^{\textup{up}}_\bslambda}
\newcommand{\decllow}{\textup{d}^{\textup{low}}_\bslambda}

\newcommand{\fu}{{f_{\bsu}}}

\newcommand{\gmu}{\gamma_{\bsu}}
\newcommand{\gmv}{\gamma_{\bsv}}

\newcommand{\gv}{{g_{\bsv}}}
\newcommand{\ku}{k_{\bsu}}
\newcommand{\luj}{\lambda_{\bsu,\bsj}}
\newcommand{\rhou}{{\rho^{\bsu}}}
\newcommand{\Au}{{A_{\bsu}}}

\newcommand{\Du}{{D^{\bsu}}}
\newcommand{\Hghc}{{H_{\widehat{\bsgamma}_c}}}

\newcommand{\Su}{S_{\bsu}}
\newcommand{\Wu}{{W_{\bsu}}}

\newcommand{\UU}{{\mathcal{U}}}
\newcommand{\UUg}{\UU_{\bsgamma}}

\newcommand{\Mgeps}{M_{\bsgamma}(\varepsilon)}
\newcommand{\Mgheps}{M_{\ghc}(\varepsilon)}
\newcommand{\Mbar}{\overline{M}}

\newcommand{\gmh}{{\widehat{\gamma}}}
\newcommand{\ghc}{\widehat{\bsgamma}_c}

\newcommand{\cost}{\textup{cost}}
\newcommand{\comp}{\textup{comp}}

\newcommand{\AOpt}{\A_{\varepsilon,\bsgamma}^{\textup{opt}}}
\newcommand{\AOptc}{\A_{\varepsilon,\ghc}^{\textup{opt}}}
\newcommand{\meps}{m_{\varepsilon}}

\DeclareMathOperator{\decay}{decay}

\newcommand{\X}{{\mathfrak{X}}}

\newcommand{\W}{\mathcal{W}}
\newcommand{\Sol}{\mathcal{S}}
\newcommand{\A}{\mathcal{A}}

\newcommand{\fh}{\widehat{f}}

\newcommand{\Kg}{K_\gamma}

\newcommand{\ConvAllInf}{r_{\infty}}

\makeindex             % used for the subject index
\begin{document}

\title*{Infinite-Variate $L^2$-Approximation with Nested Subspace Sampling}
\author{Kumar Harsha, Michael Gnewuch, and Marcin Wnuk}
\institute{K. Harsha \textbullet  M.~Gnewuch \textbullet M. Wnuk \at Institute for Mathematics, Osnabr\"uck University, Albrechtra{\ss}e 28a, 49074 Osnabr\"uck, Germany\\\email{kumar.harsha@uni-osnabrueck.de}
\and M.~Gnewuch \at \email{michael.gnewuch@uni-osnabrueck.de}
\and M. Wnuk \at \email{marcin.wnuk@uni-osnabrueck.de}}
\maketitle

\abstract*{We consider $L^2$-approximation on weighted reproducing kernel Hilbert spaces of functions depending on infinitely many variables. We focus on unrestricted linear information, admitting evaluations of arbitrary continuous linear functionals. We distinguish between ANOVA and non-ANOVA spaces, where, by ANOVA spaces, we refer to function spaces whose norms are induced by an underlying ANOVA function decomposition. In ANOVA spaces, we provide an optimal algorithm to solve the approximation problem using linear information. We determine the upper and lower error bounds on the polynomial convergence rate of $n$-th minimal worst-case errors, which match if the weights decay regularly. For non-ANOVA spaces, we also establish upper and lower error bounds. Our analysis reveals that for weights with a regular and moderate decay behavior, the convergence rate of $n$-th minimal errors is strictly higher in ANOVA than in non-ANOVA spaces.}

\abstract{We consider $L^2$-approximation on weighted reproducing kernel Hilbert spaces of functions depending on infinitely many variables. We focus on unrestricted linear information, admitting evaluations of arbitrary continuous linear functionals. We distinguish between ANOVA and non-ANOVA spaces, where, by ANOVA spaces, we refer to function spaces whose norms are induced by an underlying ANOVA function decomposition. In ANOVA spaces, we provide an optimal algorithm to solve the approximation problem using linear information. We determine the upper and lower error bounds on the polynomial convergence rate of $n$-th minimal worst-case errors, which match if the weights decay regularly. For non-ANOVA spaces, we also establish upper and lower error bounds. Our analysis reveals that for weights with a regular and moderate decay behavior, the convergence rate of $n$-th minimal errors is strictly higher in ANOVA than in non-ANOVA spaces.}

\section{Introduction}
\label{sec:1}
%Use the template \emph{chapter.tex} together with the document class SVMono (monograph-type books) or SVMult (edited books) to style the various elements of your chapter content.

%Instead of simply listing headings of different levels we recommend to let every heading be followed by at least a short passage of text.  Further on please use the \LaTeX\ automatism for all your cross-references and citations. And please note that the first line of text that follows a heading is not indented, whereas the first lines of all subsequent paragraphs are.

We want to investigate $L^2$-approximation on weighted reproducing kernel Hilbert spaces (RKHSs) of functions depending on infinitely many variables. For each of
those RKHSs, there is a family of weights and a function decomposition, which determine the norm on the space. With the help of the function decomposition, it is possible to represent an infinite-variate function as an infinite sum of functions depending only on finitely many variables.

In this paper, we focus on linear algorithms that may obtain information about an input function $f$ by evaluating finitely many continuous linear functionals in $f$. The information cost of such an algorithm is the sum of the cost of all functional evaluations, which, in turn, depends on the underlying function decomposition of the weighted RKHS we are working on.
Such a situation has already been studied in \cite{Was12,WW11a}, where the considered cost model is a generalization of the one introduced in \cite{KSWW10a} for information consisting only of function evaluation (so-called standard information). In \cite{WW11a}, a rather specific approximation problem, called the $\mathcal{G}$-approximation problem, is analyzed; there, the approximation error is measured in some very special Hilbert space $\mathcal{G}$. In \cite{Was12}, $L^2$-approximation is investigated in full generality, apart from the case that can be reduced to the $\mathcal{G}$-approximation problem.
To describe the results from \cite{Was12,WW11a} and from this paper in a more common language, we distinguish two cases of weighted RKHSs: the ones whose underlying function decomposition is an ANOVA decomposition are called ANOVA spaces, whereas the remaining RKHSs are referred to as non-ANOVA spaces (for a rigorous definition, see Section~\ref{subsed:weighted_fct_spaces}).

The case of ANOVA spaces was covered by Wasilkowski and Wo\'{z}niakowski, who in \cite{WW11a} exactly determined the polynomial convergence rates $r_\infty$ of $n$-th minimal errors (for a definition, see Section \ref{subsec:min_errors}). For the non-ANOVA spaces, Wasilkowski derived in \cite{Was12} a lower bound for $r_\infty$; with the help of auxiliary weights, he re-used the constructive results from \cite{WW11a}. He also proved a matching upper bound. Interestingly, it turns out that for slowly decaying weights, the convergence rate in the ANOVA case is superior to the one in the non-ANOVA case, while it stayed the same if the weights decrease quickly enough.

%In this paper we consider a different cost model which is less generous than in \cite{Was12, WW11a}. 
%For this cost model, which we call nested subspace sampling model, we are able to derive the exact polynomial convergence rate $r_\infty$ of $n$-th minimal errors in the ANOVA case, see Theorem~\ref{th:ANOVA_rate}.
The cost model used in \cite{Was12, WW11a} is unrestricted subspace sampling (see \RefSec{subsec:algos_cost}). In this paper, we instead consider the nested subspace sampling cost model, which is less generous than the unrestricted one. For the nested subspace sampling cost model, we have derived lower and upper bounds on the polynomial convergence rate $r_\infty$ of $n$-th minimal errors in the ANOVA case, see Theorem~\ref{th:ANOVA_rate}, where the bounds match if the weights exhibit a regular decay behavior. In the non-ANOVA case, we are able to provide upper and constructive lower bounds on $r_\infty$, see Propositions~\ref{pr:nA_low_err} and \ref{Prop:low_bound_non-ANOVA}. Although these bounds do not match in general, they show that in some cases the convergence rate in the ANOVA case is higher than the one in the non-ANOVA case, cf. Corollary~\ref{cor:comparison}. The reason for that is mainly that the cost of an algorithm may depend on what kind of function decomposition underlies the function space it is defined on. This dependence, unfortunately, prevents us from simply using the embedding machinery developed in \cite{HR13, GHHR17, GHHRW19} to transfer results for ANOVA spaces directly to spaces with a different underlying function decomposition.

% In this approach, we assume that the cost of one function evaluation depends on the number of "active" variables. For a more precise explanation, let $\bsx = [x_1,x_2,\dots]$ be an infinite-variate point, and we want to compute $f(\bsx)$. Then, the "active" variables in $\bsx$ are all non-zero components $x_j$, so the cost depends on the size of the active variables. In this paper, we instead consider the nested subspace sampling cost model, where the cost of evaluating $f(\bsx)$ depends on the importance (or weights) of the variables in the active set. More specifically, including lower-importance variables in the active set increases the corresponding evaluation cost. Consequently, the unrestricted sampling model is more generous than the nested one.

%This problem has been studied in \cite{Was12,WW11a,WW11b}, 
%where the upper error bounds have been achieved using multivariate decomposition methods (fka changing dimension algorithms). 
%We analyze the problem for a different cost model that is for some applications more appropriate.

\textbf{Notation} For a set $\bsu$, we denote by $\abs{\bsu}$ its cardinality. We use the shorthand $[k] := \{1,\dots,k\}$ for every $k \in \N$.

For two non-negative functions $f,g : D \to \R$, we use $f \lesssim g$ to denote that there exists a constant $C > 0$ such that $f(x) \leq Cg(x)$ for all $x \in D$. Likewise, we use $f \gtrsim g$ to denote that there exists $c > 0$ such that $f(x) \geq cg(x)$. Finally, we use $f \asymp g$ to denote that there exist two constants $c_1,c_2 > 0$ such that $c_1g(x) \leq f(x) \leq c_2g(x)$ for all $x$. The notation also applies to sequences.

\section{Function Spaces and  Algorithms}\label{sec:fn_spaces}

For general results on reproducing kernel Hilbert spaces (RKHS) we refer to \cite{Aro50}, and to \cite{GMR14} for results on the function spaces discussed below.

\subsection{Weighted Function Spaces}\label{subsed:weighted_fct_spaces}

Let $D\neq \emptyset$ be an arbitrary set, $\Sigma$ a $\sigma$-algebra on $D$, and $\rho: \Sigma \to [0,1]$  a probability measure on $D$. 
Let $k:D\times D\to \R$ be a reproducing kernel on $D$ such that 
\begin{equation*}
    H(1) \cap H(k)= \{0\};
    %\hspace{3ex}\text{and}\hspace{3ex} 
    %H(1+k) \subseteq L_2(D,\rho, \R);
\end{equation*} 
here and in the whole paper we follow the convention to denote by $H(k)$ the uniquely determined Hilbert space with reproducing kernel $k$ and its norm by $\|\cdot \|_{H(k)}$; in particular $H(1)$ is the space of all  constant functions on the corresponding domain $D$. Put 
$$ \UU := \{ \bsu \subset \N \,|\, \abs{\bsu} < \infty \}, $$ 
where $\abs{\bsu}$ refers to the cardinality of $\bsu$. For $\bsu \in \UU$ we define $\ku: D^\N \times D^\N\to \R$ by 
\begin{equation*}
    \ku(\bsx, \bsy) := \prod_{j\in \bsu} k(x_j, y_j) \hspace{3ex}\text{for all $\bsx, \bsy \in D^\N$.}
\end{equation*}
Since $\ku$ only depends on the variables $x_j$, $j\in \bsu$, we may alternatively view it as a function on $\Du$ by putting
$$
\ku(\bsx_\bsu, \bsy_\bsu) := \ku(\tilde{\bsx}, \tilde{\bsy})
$$
for arbitrary $\tilde{\bsx}, \tilde{\bsy}\in D^\N$ satisfying $\tilde{\bsx}_\bsu = \bsx_\bsu$, $\tilde{\bsy}_\bsu = \bsy_\bsu$, where $\bsx_\bsu := \{ x_j\}_{j \in \bsu} \in \Du$. We use an analogous convention for arbitrary functions on $D^\N$ that only depend on variables $x_j$, $j\in \bsu$.
We denote the RKHS with kernel $\ku$ by $\Hu := H(\ku)$ and its norm by $\|\cdot \|_{\Hu}$; in particular, we have $H_\emptyset = H(1)$ due to the convention that an empty product is equal to $1$.

For a family $\bsgamma = \{ \gmu \}_{\bsu \in \UU}$ of non-negative weights, we put
$$ \UU_\bsgamma := \{ \bsu \in \UU \,|\, \gmu >0\}. $$
We work with product weights which are given by a decreasing sequence of weights $1 \geq \gamma_1 \geq \gamma_2 \geq ... \geq 0$ assigned to all the variables. Then the weight (or importance) of an interacting set of variables $\bsu \in \UUg$ is given by $\gamma_\bsu = \prod_{j\in \bsu} \gamma_j$.
Moreover, we define
\begin{equation*}
    \X := \bigg\{ \bsx \in D^{\N}\,\bigg|\, \sum_{\bsu \in \UUg} \gmu k_\bsu(\bsx,\bsx) < \infty \bigg\}
\end{equation*}
and the reproducing kernel $K_\bsgamma: \X \times \X \to \R$ by
\begin{equation} \label{eq:Kg_decomp}
    K_\bsgamma(\bsx,\bsy) := \sum_{\bsu \in \UUg} \gmu \ku(\bsx,\bsy) \hspace{2ex}\bsx,\bsy\in \X.
    %\prod^\infty_{j=1} (1+\gamma_j k(x_j,y_j)) 
\end{equation}
We denote the corresponding RKHS by $\Hg := H(K_\bsgamma)$. $\Hg$ decomposes into an orthogonal sum of subspaces
\begin{equation}\label{eq:Hg_decomp}
    \Hg = \bigoplus_{\bsu \in \UUg} H(\gmu k_\bsu). %= \bigoplus_{\bsu \in \UUg} \gmu \Hu.
\end{equation}
Note that for $\bsu \in \UUg$ we have that $H(\gmu \ku)$ and $\Hu$ are equal as vector spaces, but their norms differ by a factor $\gmu^{-1/2}$. Hence, the norm in $\Hg$ is given by 
\begin{equation*}
    \norm{f}_{\Hg}^2 = \sum_{\bsu \in \UU_\bsgamma} \gmu^{-1} \norm{\fu}_{\Hu}^2,
\end{equation*}
where 
\begin{equation}\label{eq:fn_decomp}
    f = \sum_{\bsu \in \UU_\bsgamma} \fu, \hspace{3ex} \fu \in \Hu,
\end{equation}
is the uniquely defined function decomposition of $f$ in $\Hg$.

% Let $ C_0 := \sup_{f \in H(k)} \frac{\norm{f}_{L^2(D,\rho)}}{\norm{f}_{H(k)}}. $
We assume 
\begin{equation}\label{eq:assump1}
    \int_D k(x,x) \,\drho(x) < \infty.
\end{equation}
It is well-known that \eqref{eq:assump1} implies that $S : H(1+k) \to L^2(D,\rho)$ is a compact operator. We denote the corresponding operator norm by $\norm{S}$. From this, we also get the compact mapping $\Hu \to L^2(\Du, \rhou)$ for all $\bsu \in \UU_\bsgamma$. Furthermore, we assume
\begin{equation}\label{eq:assump2}
    \lim_{j \to \infty} \gamma_{\bsu_j} \norm{S}^{2\abs{\bsu_j}} = 0,
\end{equation}
where $\{ \bsu_j \}_{j \in \N}$ is some enumeration of the elements of $\UUg$. It follows from \eqref{eq:assump2} that the map $\Hg \to L^2(\X, \rho^\N)$ is also compact, see \cite{Wnuk22}[Theorem~2]. In particular, the \emph{solution operator} we consider
\begin{equation*}
    \Sol: \Hg \to L^2(\X, \rho^\N)\,,\, f \mapsto f
\end{equation*}
is well-defined and continuous.
% Let $$P_\bsu: \Hg \to H(\gmu k_\bsu)$$ denote an orthogonal projection operator.

\begin{definition}[ANOVA kernels]\label{def:ANOVA_kern}
    Let $k : D \times D \to \R$ be a reproducing kernel such that $k(\cdot,x) \in L^2(D,\rho)$ for all $x \in D$. If the kernel satisfies the \emph{ANOVA condition}
    \begin{equation}\label{eq:anova_cond}
        \int_D k(y,x) \drho(y) = 0 \quad \forall x \in D ,
    \end{equation}
    then we call $k$ an \emph{ANOVA kernel}. Otherwise, we call it a \emph{non-ANOVA kernel}.
\end{definition}

\begin{remark}\label{re:ANOVA}
    \begin{enumerate}
        \item It is easily seen that under the assumptions of Definition \ref{def:ANOVA_kern}, the ANOVA condition \eqref{eq:anova_cond} is equivalent to the condition that $\int_D f(y) \rho(y) = 0$ for all $f \in H(k)$.
        
        \item Consider the reproducing kernel as in \eqref{eq:Kg_decomp} based on an ANOVA kernel $k:D\times D \to \R$. Then for each $f \in \Hg$, the function decomposition as in \eqref{eq:fn_decomp} is the infinite-variate ANOVA decomposition, see \cite[Remark~2.12]{BG14}. For instance, it implies that for each $f \in \Hg$, for all $\bsu \in \UU$ and $\bsx \in D^{\N}$
        \begin{equation*}
            \int_D \fu (\bsx) \,\drho(x_j) = 0 \hspace{3ex}\text{if $j\in u$.}
        \end{equation*}
    \end{enumerate}
\end{remark}

Examples of ANOVA spaces include Korobov spaces, see e.g. \cite[Appendix~A.1]{NW08} and 
unanchored Sobolev spaces, see e.g. \cite[Appendix~A.2.1,A.2.3]{NW08},\cite[Section~5]{BG14}, and \cite[Section~5]{DG14b}.
%Hermite spaces, Haar-Wavelet spaces. 
A notable example of non-ANOVA spaces are anchored Sobolev spaces, see e.g. \cite[Appendix~A.2.2]{NW08} or \cite{KSWW10a}.

\subsection{Algorithms, Cost Models, and Errors}\label{sec:ACME}

\subsubsection{Admissible algorithms and cost}\label{subsec:algos_cost}

We know that non-adaptive algorithms are as good as adaptive algorithms in the class of all deterministic algorithms for linear problems $S : F \to G$ where $F$ is a pre-Hilbert space, see \cite[Theorem~4.5]{NW08}; for a precise definition of deterministic algorithms, see \cite{TWW88}. Moreover, we know that linear algorithms are optimal for the aforementioned linear problems, see \cite[Theorem~4.8]{NW08}. Therefore, we may confine ourselves to linear algorithms using arbitrary linear information of the form
\begin{equation}\label{eq:lin_algo_all}
   \A(f) = \sum_{j=1}^{n} L_{\bsu_j} (f) \cdot g_j
\end{equation}
for some $n \in \N$, $g_j \in \Ltwo$, index sets $\bsu_j \in \UUg$, and linear functionals $L_{\bsu_j} \in \Hg^\ast$ such that $L_{\bsu_j}\vert_{\Hv}=0$ if $\bsv \nsubseteq \bsu_j$, see \cite{TWW88}. Owing to the Riesz representation theorem, we can define the linear functional $L_{\bsu_j}$ as inner product with a representer as
\begin{equation}\label{eq:info_fnl}
    L_{\bsu_j}(f) := \innProd{f}{\eta_{\bsu_j}}{\Hg}, \quad \text{where } \eta_{\bsu_j} \in \bigoplus_{\bsv \subseteq \bsu_j} H_{\bsv}.
\end{equation}

The worst-case error of the algorithm $\A$ for $L^2$-approximation with respect to $\Hg$ is given as
\begin{equation*}
   e(\A, \Hg) := e(\A, \Sol, \Hg) := \sup_{\|f\|_{\Hg} \le 1} \| \Sol(f)- \A(f) \|_{L^2(\X, \rho^\N)}.
\end{equation*}

Let us fix a cost function $\$ : \N_0 \to [1,\infty)$ which is non-decreasing. Consider a linear functional $L_\bsu$ given by $L_\bsu := \innProd{\cdot}{\eta_\bsu}{\Hg}$, where $\eta_\bsu \in \bigoplus_{\bsv \subseteq \bsu} H_{\bsv}$ but there exists no $\bsw \subsetneq \bsu$ such that $\eta_\bsu \in \bigoplus_{\bsv \subseteq \bsw} H_{\bsv}$. Then the cost of $L_\bsu$ in the \emph{nested subspace sampling} (NSS) model (c.f. \cite[Sect.~5.1]{GHHR17} or \cite{MR09, HMNR10, BG14, DG14a}) is given by
\begin{equation*}
    \cost (L_\bsu) = \$(\max \bsu),
\end{equation*}
where, by convention, $\max \emptyset = 0$. Furthermore, the cost of an algorithm like \eqref{eq:lin_algo_all} in the NSS model is given by
\begin{equation*}
    \cost(\A) = \sum_{j=1}^n \cost(L_{\bsu_{j}}).
\end{equation*}

For comparison, the cost of evaluating a linear functional in the unrestricted subspace sampling model as in \cite{WW11a,Was12} depends on the number of "active variables". More precisely, for the functional $L_\bsu$ considered above, the corresponding cost is given by
$$ \cost^\ast(L_\bsu) = \$(\abs{\bsu}). $$

\subsubsection{Minimal errors and complexity}\label{subsec:min_errors}

For the results in this paper, we work with the NSS cost model unless specified otherwise. On that basis, we define the \emph{$n$-th minimal error} for $n \in \N$ as
\begin{equation*}
    e(n,\Hg) := \inf \{ e(\A,\Hg)\ :\ \A \text{ as in \eqref{eq:lin_algo_all}} \text{ and } \cost(\A) \leq n \}.
\end{equation*}
The approximation problem is said to be \emph{strongly tractable} if and only if there exist non-negative constants $c,\alpha > 0$ such that for all $n \in \N$, it holds that
$$ e(n,\Hg) \leq c \cdot n^{-\alpha}. $$
In order to quantify convergence rates, we introduce the notion of decay rates. The \emph{lower polynomial decay rate} of a null sequence of positive reals $\bsx := \{ x_j \}_{j \in \N}$ is given by
\begin{equation}\label{def:decay}
    \declow_{\bsx} := \decay(\bsx) = \sup \left\{ \alpha \geq 0 : \sum_{j \in \N} x_j^{1/\alpha} < \infty \right\}.
\end{equation}
This leads to the definition of the \emph{polynomial convergence rate of $n$-th minimal errors}, c.f. \cite[Sec~4]{GHHR17},
\begin{equation}\label{eq:conv_rate}
    \ConvAllInf := \ConvAllInf(\Hg) = \decay(\{e(n,\Hg)\}_{n \in \N}),
\end{equation}
which is the quantity of interest in later sections. Since $e(n,\Hg)$ forms a non-increasing sequence, the definition \eqref{eq:conv_rate} is equivalent to
$$ \ConvAllInf(\Hg) = \sup \left\{ \alpha \geq 0 : \lim_{n \to \infty} e(n,\Hg) n^\alpha = 0 \right\}, $$
as mentioned without a proof in \cite{FHW12}; for the reader's convenience, we present a proof in \RefLem{le:dec_low_equiv}. We define the \emph{upper polynomial decay rate} of a null sequence of positive reals $\{x_n\}_{n \in \N}$ as
\begin{equation}\label{def:decay_up}
    \decup_{\bsx} := \inf \left\{ \alpha \geq 0 : \inf_{n \in \N} x_n n^\alpha > 0  \right\}.
\end{equation}
We show that $\declow_{\bsx} \leq \decup_{\bsx}$ in general, see \RefLem{le:dec_up_low}. Although one can construct sequences such that $\declow_{\bsx} < \decup_{\bsx}$, equality holds for most practical choices of sequences; confer Remark \ref{re:decay_ex}. In any case, we highlight the difference in the lower and upper decay rates in all results, except \RefCol{cor:comparison}.

We now introduce complexity concepts which also appear in tractability studies as another perspective on the minimal errors. Let $\varepsilon > 0$ be the error demand. The worst case \emph{$\varepsilon$-information complexity}, also referred to as \emph{$\varepsilon$-complexity} of function approximation is defined as, c.f. \cite{WW11a},
\begin{equation*}
    \comp(\varepsilon,\Hg) := \inf \left\{ \cost(\A)\ :\ \A \text{ as in \eqref{eq:lin_algo_all}} \text{ and } e(\A,\Hg) \leq \varepsilon \right\}.
\end{equation*}
The approximation problem is strongly tractable if and only if there exist non-negative constants $c,\alpha > 0$ such that
$$ \comp(\varepsilon,\Hg) \leq c \cdot \varepsilon^{-\alpha}. $$
The \emph{exponent of tractability} is defined as
\begin{equation}\label{def:exp_tract}
    p_{\infty} := p_{\infty}(\Hg) = \inf \left\{ \alpha \geq 0:\exists C_\alpha > 0 \forall \varepsilon \in (0,1)\  \textup{comp}(\varepsilon,\Hg) \leq C_\alpha \varepsilon^{-\alpha} \right\}.
\end{equation}
We can relate the exponent of tractability to the convergence rate of the $n$-th minimal error as $ r_\infty = 1/p_{\infty}$.

\subsubsection{Univariate and Multivariate $L^2$-Approximation}

Consider the univariate solution operator $S: H(1+k) \to L^2(D,\rho), f \mapsto f$. It follows from \eqref{eq:assump1} that $S$ is compact. This means that the self-adjoint operator $W := S^\ast \circ S$ is compact and positive definite. Consequently, $W$ has a null sequence of positive eigenvalues $\bslambda := \{ \lambda_j \}_{j \in \N}$ with a corresponding set of eigenfunctions $\{ \eta_j \}_{j \in \N}$ that form a complete orthonormal system of $H(1+k)$ satisfying
\begin{equation}\label{eq:ortho_H}
    W(\eta_j) = \lambda_j \eta_j.
\end{equation}
We assume that the eigenvalues are ordered, i.e. $ \lambda_1 \geq \lambda_2 \geq \dots$.

\begin{remark}
    We provide three examples that exhibit a sharp decay rate for the eigenvalues $\bslambda$. For the approximation problem defined on a Korobov space of smoothness $\alpha$, it can be checked that the eigenvalues behave as $\lambda_j \asymp j^{-2\alpha}$; see \cite[p~341]{NW08}. On the classical Sobolev space, the eigenvalues follow $\lambda_j \asymp j^{-2}$; see \cite[p~410]{WW11a}. In both of these cases, $\decup_{\bslambda} = \declow_{\bslambda}$.

    Another example, which recently captured a lot of interest, is Hermite spaces with the reproducing kernel $k(x,y) = \sum_{\nu \in \N_0} \alpha_\nu^{-1} h_\nu(x) h_\nu(y)$ for $x,y \in \R$, where the Fourier weights $\alpha_\nu$ are essentially of the form $\alpha_\nu := (\nu+1)^r$ for some $r > 1/2$ (PG), or $\alpha_\nu := 2^{r\nu^b}$ for some $b,r > 0$ (EG). Furthermore, $h_\nu$ denotes the $L^2$-normalized Hermite polynomial of degree $\nu$. The first paper considering spaces of this type was \cite{IKP16}; for more recent results, see \cite{GHRR23} and the literature mentioned therein. Here, the eigenvalues are $\lambda_j = \alpha_{j+1}^{-1}$ for $j \in \N$; see, e.g. \cite[Section~4.1.2]{GHRR23}. Hence, we have $\declup = r = \decllow$ in the case (PG) and $\declup = \infty = \decllow$ in the case (EG).
\end{remark}

Let $\declow_\bslambda$ denote the lower decay rate of the eigenvalues $\bslambda$ of $W$. We assume that $\declow_\bslambda > 1$ so that \eqref{eq:assump1} is satisfied, since $\int_D k(x,x) \,\drho(x) = \sum_{j \in \N} \lambda_j$. Similarly, as in \eqref{def:decay}, we may define the convergence rate of the $n$-th minimal errors of the univariate approximation problem as
\begin{equation*}
    r_1 := r_1(H(1+k)) = \decay(\{e(n,H(1+k))\}_{n \in \N}).
\end{equation*}
We also know that $e(n,H(1+k)) = \sqrt{\lambda_{n+1}}$, which means, c.f. \cite[p~314]{Was12},
$$r_1 = \decllow/2.$$

Let $\Su : \Hu \to L^2(D^{\bsu},\rhou)$ be given by $\Su(f) = f$ for all $f \in \Hu$. Owing to the tensor product construction of spaces $\Hu$ for $\bsu \in \UUg$, the eigenpairs of the operator $\Wu := \Su^\ast \circ \Su : \Hu \to \Hu$ are of product form. In particular, for $\bsu = \{ u_1, u_2, \dots, u_k \} \in \UU$ with $k = \abs{\bsu}$, and for all $\bsj = ( j_1, j_2, \dots, j_k ) \in \N^k$, we define
\begin{equation*}
    \lambda_{\bsu,\bsj} := \prod_{i=1}^k \lambda_{j_i} \quad \text{and} \quad \eta_{\bsu,\bsj}(\bsx) := \prod_{i=1}^k \eta_{j_i}(x_{u_i}),
\end{equation*}
using the eigenvalues and eigenfunctions of the univariate operator as in \eqref{eq:ortho_H}. Then
\begin{equation}\label{eq:ortho_Hu}
	\Wu(\eta_{\bsu,\bsj}) = \lambda_{\bsu,\bsj} \eta_{\bsu,\bsj} ,
\end{equation}
and $\innProd{\eta_{\bsu,\bsi}}{\eta_{\bsu,\bsj}}{\Hu} = \delta_{\bsi,\bsj}$ for $\bsi, \bsj \in \N^{k}$. Using the eigenpairs of $\Wu$, we can express the SVD of $\Su$ as
\begin{equation}\label{eq:svd_Su}
    \Su (\fu) = \sum_{\bsj \in \N^{\bsu}} \sqrt{\lambda_{\bsu,\bsj}} \innProd{\fu}{\eta_{\bsu,\bsj}}{\Hu} \cdot \bar{\eta}_{\bsu,\bsj},
\end{equation}
where the
\begin{equation}\label{eq:etaBar}
    \bar{\eta}_{\bsu,\bsj} = \Su(\eta_{\bsu,\bsj})/\sqrt{\lambda_{\bsu,\bsj}}, \quad \bsj \in \N^k,
\end{equation}
form an orthonormal system in $\Ltwou$. In the same way as for the spaces $\Hu$, we now define a self-adjoint operator for the infinite-variate setting $$\W := \Sol^\ast \circ \Sol : \Hg \to \Hg$$ which is compact and positive definite.

\section{ANOVA spaces}\label{sec:ANOVA}

For notational clarity, we use $\Kg^\textup{A}$ and $\HgA$ to refer to the infinite-variate ANOVA kernel and the corresponding function space in this section. It is well known that in ANOVA spaces we have, see, e.g. \cite{BG14},
\begin{equation}\label{eq:anova_finite}
    \innProd{\fu}{\gv}{L^2(D^{\bsu \cup \bsv}, \rho^{\bsu \cup \bsv})} = 0 \quad\text{for } \bsu,\bsv \in \UU, \bsu \neq \bsv, \fu \in \Hu, \gv \in \Hv.
\end{equation}

\begin{lemma}\label{le:err_decomp_A}
    In the ANOVA setting, any linear continuous algorithm $\A : \HgA \to \Ltwo$ satisfying $\A(\Hu) \subseteq \Sol(\Hu) = \Hu$ for all $\bsu \in \UUg$ also satisfies
    $$ \norm{\Sol(f)-\A(f)}_{\Ltwo}^2 = \sum_{\bsu \in \UUg} \norm{S_{\bsu}(\fu) - A_{\bsu}(\fu)}_{\Ltwou}^2, $$
    where $\Au : \Hu \to \Ltwou$ and $\Au(f) = \A(f)$ for all $f \in \Hu$.
\end{lemma}
\begin{proof}
    The statement follows directly from \eqref{eq:fn_decomp} and \eqref{eq:anova_finite}, using the fact that $\Sol(\fu) = \Su(\fu)$.
\end{proof}
Analogous to \cite{WW11a}, we define an approximation algorithm based on the SVD of $\Su$ from \eqref{eq:svd_Su}. Let $\delta \in (0,1)$ and define for $f \in \Hu$
\begin{equation}\label{eq:opt_algo_u}
    A_{\bsu,\delta}^\ast(f) := \sum_{\bsj \in J(\delta,\bsu)} \sqrt{\lambda_{\bsu,\bsj}} \innProd{f}{\eta_{\bsu,\bsj}}{\Hu} \cdot \bar{\eta}_{\bsu,\bsj},
\end{equation}
where $J(\delta,\bsu) := \{ \bsj : \sqrt{\luj} > \delta \}$ and $\bar{\eta}_{\bsu,\bsj}$ is given by \eqref{eq:etaBar}. We want to combine all the sub-algorithms $A_{\bsu,\delta}^\ast : \Hu \to \Ltwou$ for $\bsu \in \UUg$ to attack the approximation problem on $\Hg$. Let $\varepsilon \in (0,1)$ be the worst-case error threshold for the infinite-variate problem, and define the set
\begin{equation}\label{eq:opt_set}
    \begin{aligned}[b]
        \Mgeps &:= \bigcup_{\bsu \in \UUg} \{ (\bsu, \bsj) : \bsj \in J(\varepsilon/\sqrt{\gmu}, \bsu) \}\\
        &= \left\{(\bsu,\bsj): \bsu \in \UUg, \bsj \in \N^{\bsu}, \text { and } \gmu \lambda_{\bsu,\bsj}>\varepsilon^{2}\right\}.
    \end{aligned}
\end{equation}
We now define the algorithm $\AOpt : \HgA \to \Ltwo$ as
\begin{equation}\label{eq:opt_algo}
    \AOpt(f) := \sum_{\bsu \in \UUg} A_{\bsu,\varepsilon/\sqrt{\gmu}}^\ast(\fu) = \sum_{(\bsu,\bsj)\in M_{\bsgamma}(\varepsilon)} \innProd{f}{\eta_{\bsu,\bsj}}{\Hu} \cdot \Su(\eta_{\bsu,\bsj}).
\end{equation}

\begin{lemma}\label{le:opt_err_u}
    For any $\bsu \in \UUg$ and $\delta > 0$, the algorithm $A_{\bsu,\delta}^\ast$ from \eqref{eq:opt_algo_u} satisfies $\norm{\Su - A_{\bsu,\delta}^\ast} \leq \delta$.
\end{lemma}
\begin{proof}
    Let $\delta \in (0,1)$. Then using \eqref{eq:svd_Su} and Bessel's inequality, we get
    \begin{equation*}
        \begin{aligned}
            &\norm{\Su(f) - A_{\bsu,\delta}^\ast(f)}^2_{\Ltwou} = \norm{ \sum_{\bsj \notin J(\delta,\bsu)} \sqrt{\lambda_{\bsu,\bsj}} \innProd{f}{\eta_{\bsu,\bsj}}{\Hu} \cdot \bar{\eta}_{\bsu,\bsj} }^2_{\Ltwou}\\
            &= \sum_{\bsj \notin J(\delta,\bsu)} \lambda_{\bsu,\bsj} \abs{\innProd{f}{\eta_{\bsu,\bsj}}{\Hu}}^2 \norm{\bar{\eta}_{\bsu,\bsj}}^2_{\Ltwou} \leq \delta^2 \sum_{\bsj \notin J(\delta,\bsu)} \abs{\innProd{f}{\eta_{\bsu,\bsj}}{\Hu}}^2 \leq \delta^2 \norm{f}^2_{\Hu}.
        \end{aligned}
    \end{equation*}
    The result now follows by taking supremum over $f \in \Hu$ with $\norm{f}_{\Hu} = 1$.
\end{proof}

The next result is reformulated from \cite[Theorem~1(i)]{WW11a}, and has a straightforward proof.

\begin{corollary}\label{th:opt_err}
    In the ANOVA setting, the algorithm $\AOpt$ as in \eqref{eq:opt_algo} achieves a worst case error at most $\varepsilon$.
\end{corollary}
The following result was mentioned in \cite{WW11a}\label{lemma:decompW} without a proof, therefore we present it here for the reader's convenience.
\begin{lemma}
	The operator $\W=\Sol^{\ast}\Sol$ admits the following decomposition over ANOVA spaces
	\begin{equation}
		\W(f) = \sum_{\bsu \in \UUg} \gamma_{\bsu} \Wu(\fu) \quad \forall f \in \HgA,\ f = \sum_{\bsu \in \UUg} \fu .
	\end{equation}
\end{lemma}

\begin{proof}
    The Definition \ref{def:ANOVA_kern} and the dominated convergence theorem imply that
    \begin{equation*}
        \begin{aligned}
            (\W f)(\bsy) &= \innProd{\W(f)}{\Kg(\cdot,\bsy)}{\HgA} = \innProd{\Sol(f)}{\Sol(\Kg(\cdot,\bsy))}{\Ltwo}\\
            &= \int_{\X} \Kg(\bsx,\bsy) f(\bsx) \drho^{\N}(\bsx) = \int_{\X} \sum_{\bsu \in \UUg} \gmu \ku(\bsx_\bsu,\bsy_\bsu) f(\bsx) \drho^{\N}(\bsx) \\
            &= \sum_{\bsu \in \UUg} \gmu \int_{\Du} \ku(\bsx_\bsu,\bsy_\bsu) \fu(\bsx_\bsu) \drho^{\bsu}(\bsx_\bsu) = \sum_{\bsu \in \UUg} \gmu (\Wu \fu) (\bsy_\bsu).
        \end{aligned}
    \end{equation*}
    The result now follows.
\end{proof}

Using the eigenpairs from \eqref{eq:ortho_Hu} and Lemma \ref{lemma:decompW}, we can conclude that the eigenpairs of $\W$ with positive eigenvalues are
\begin{equation}\label{W_eig}
	\left\{ \left( \gmu \lambda_{\bsu,\bsj}, \xi_{\bsu,\bsj} \right) : \bsu \in \UUg, \bsj \in \N^{\bsu} \right\} ,
\end{equation}
with eigenfunctions $\xi_{\bsu,\bsj} := \sqrt{\gmu} \eta_{\bsu,\bsj}$ normalized in $\HgA$.
We can now express the solution operator $\Sol$ using its singular values and the orthonormal system  as
\begin{equation*}
    \begin{aligned}
        \Sol(f) = \sum_{\bsu \in \UUg} \sum_{\bsj \in \N^{\bsu}} \innProd{f}{\eta_{\bsu,\bsj}}{\Hu}\cdot \Sol(\eta_{\bsu,\bsj}) = \sum_{\bsu \in \UUg} \sum_{\bsj \in \N^{\bsu}} \sqrt{\gmu \lambda_{\bsu,\bsj}}\ \innProd{f}{\xi_{\bsu,\bsj}}{\Hg}\cdot \bar{\eta}_{\bsu,\bsj},
    \end{aligned}
\end{equation*}
where $\bar{\eta}_{\bsu,\bsj}$ is as defined in \eqref{eq:etaBar}.

The next result is the analogue of \cite[Theorem~1(iii)]{WW11a} where it was proven for unrestricted subspace sampling. It can be checked that the proof follows verbatim and the arguments hold irrespective of the cost model.

\begin{theorem}\label{th:opt_cost}
    In the nested subspace sampling model, the cost of the algorithm $\AOpt$ is minimal amongst all algorithms that achieve a worst case error of at most $\varepsilon$.
\end{theorem}

Due to \RefThm{th:opt_cost}, $\AOpt$ from \eqref{eq:opt_algo} is the optimal algorithm for the approximation problem in ANOVA function spaces. Therefore, it suffices to compute the cost of $\AOpt$ to determine the $\varepsilon$-complexity. Let
\begin{equation}\label{eq:opt_set_k}
    \begin{aligned}
        M(\varepsilon,k) &:= \left\{(\bsu,\bsj): k \in \bsu \subseteq [k], \bsj \in \N^{\bsu}, \gamma_{\bsu} \lambda_{\bsu,\bsj}>\varepsilon^{2}\right\}, \quad \text{and}\\
        \quad n_k &:= \abs{M(\varepsilon,k)}.
    \end{aligned}
\end{equation}
In order to simplify the cost analysis later on, we can rewrite the optimal algorithm from \eqref{eq:opt_algo} in a multilevel fashion as
\begin{equation}\label{eq:opt_algo_ml}
    \AOpt (f) = \sum_{k=1}^{\meps} \sum_{(\bsu,\bsj) \in M(\varepsilon,k)} \innProd{f}{\eta_{\bsu,\bsj}}{\Hu} \cdot \Sol (\eta_{\bsu,\bsj}),
\end{equation}
where $\meps \in \N$ is the largest number $k$ such that $n_k > 0$.

\begin{lemma}\label{lem:m_eps_UL}
    For an error threshold $\varepsilon > 0$, the optimal algorithm $\AOpt$ samples from $\meps$ levels which can be estimated as
        $$ \meps \gtrsim \varepsilon^{-2/p_1} \quad \forall p_1 \in (\decgup,\infty), $$
        and
        $$ \meps \lesssim \varepsilon^{-2/p_2} \quad \forall p_2 \in (0,\decglow). $$
\end{lemma}
\begin{proof}
    We have for all $p_1 \in (\decgup, \infty)$ and $p_2 \in (0,\decglow)$
    \begin{equation}\label{eq:gamma_bnds}
        \gamma_j \gtrsim j^{-p_1} \quad \text{and} \quad \gamma_j \lesssim j^{-p_2},
    \end{equation}
    see \RefLem{le:decay_bounds}. We have $(\{\meps+1\},1) \notin \Mgeps$ which means $\gamma_{\meps+1}\lambda_{1} \leq \varepsilon^2$. Using the lower bound from \eqref{eq:gamma_bnds}, we get for some $c_1 > 0$
    $$ c_1 \lambda_1 (\meps+1)^{-p_1} \leq \lambda_1 \gamma_{\meps+1} \leq \varepsilon^2 \implies (c_1 \lambda_1)^{1/p_1} \varepsilon^{-2/p_1} \leq \meps + 1. $$
    Therefore, we get
    \begin{equation}\label{eq:m_ceil}
        \meps \geq (c_1 \lambda_1)^{1/p_1} \varepsilon^{-2/p_1} - 1,
    \end{equation}
    and the lower bound from the lemma follows.
    
    To show the remaining upper bound from the lemma, we note that $\bigl(\{\meps\},1\bigr) \in \Mgeps$ must hold, which is equivalent to the condition $\lambda_1 \gamma_{\meps} > \varepsilon^2$. Then using \eqref{eq:gamma_bnds} we get for some constant $c_2 > 0$
    \begin{equation}\label{eq:m_floor}
        \varepsilon^2 < \lambda_1 \gamma_{\meps}  \leq c_2 \lambda_1 \meps^{-p_2} \implies \meps \leq (c_2 \lambda_1)^{1/p_2} \varepsilon^{-2/p_2},
    \end{equation}
    giving us the desired upper bound.
\end{proof}

We now want to estimate the cardinality of $M(\varepsilon,k)$ as defined in \eqref{eq:opt_set_k}. Before we proceed, we introduce the following notation. Let $(\bsu_1,\bsj_1), (\bsu_2,\bsj_2) \in \Mgeps$ such that $\bsu_1 \cap \bsu_2 = \emptyset$, $\abs{\bsu_1} = k_1, \abs{\bsu_2} = k_2$ and $\bsj_1 \in \N^{k_1}, \bsj_2 \in \N^{k_2}$. Then, we define $(\bsu_1, \bsj_2) \times (\bsu_2, \bsj_2) = (\bsu, \bsj)$ with $\bsu = \bsu_1 \cup \bsu_2, \bsj \in \N^{\bsu_1 \cup \bsu_2}$ such that $j_i = j_{1,i}$ if $i \in \bsu_1$, and $j_i = j_{2,i}$ if $i \in \bsu_2.$
Moreover, we define $\Mbar(\varepsilon,k) := \left\{ (k,j) : j \in \N, \gamma_k\lambda_j > \varepsilon^2 \right\}$, and $\overline{n}_k := \abs{\Mbar(\varepsilon,k)}$.

It is straightforward to show that for every $\tau \in (0,\decglow)$ and every constant $C_\tau > 0$
\begin{equation}\label{eq:prodWt}
    \sum_{k \in \bsu \subseteq [k]} \gmu^{1/\tau} C_\tau^{\abs{\bsu}} \asymp \gamma_k^{1/\tau};
\end{equation}
confer \cite[Theorem~5 and Corollary~1]{DG14a} where the statement was proved for more general weights.

\begin{lemma}\label{lem:M_eps_k}
    Let $\varepsilon \in (0,1)$ and $k \in [\meps]$. Then for all $q_1 \in (\declup,\infty)$ and $q_2 \in (1,\decllow)$ we have the following properties.
    \begin{enumerate}
        \item $M(\varepsilon,k+1) = \Mbar(\varepsilon,k+1) \cup  \bigcup_{\ell=1}^k \bigcup_{j \in \N} \left\{ (\{k+1\}, j) \times M\left(\frac{\varepsilon}{\sqrt{\gamma_{k+1} \lambda_j}},\ell \right) \right\}$. \label{item1}
        \item $\abs{M(\varepsilon,k+1)} = \abs{\Mbar(\varepsilon,k+1)} + \sum_{\ell = 1}^{k} \sum_{j \in \N} \abs{M\left(\frac{\varepsilon}{\sqrt{\gamma_{k+1} \lambda_j}},\ell\right)}$. \label{item2}
        \item $\abs{\Mbar(\varepsilon,k)} \gtrsim \gamma_k^{1/q_1} \varepsilon^{-2/q_1}$.\label{item3}
        \item $\abs{\Mbar(\varepsilon,k)} \lesssim \gamma_k^{1/q_2} \varepsilon^{-2/q_2}$. \label{item4}
        \item $\abs{M(\varepsilon,k)} \lesssim \varepsilon^{-2/q_2} \sum_{k \in \bsu \subseteq [k]} \gmu^{1/q_2} c(\bslambda,q_2)^{\abs{u}-1}$, \\ where $c(\bslambda,q_2) = \sum_{j \in \N} \lambda_j^{1/q_2} < \infty$. \label{item5}
        \item $\abs{M(\varepsilon,k)} \lesssim \varepsilon^{-2/q_2} \gamma_k^{1/q_2}$, if additionally $q_2 \in (1,\min\{ \decllow,\decglow \})$. \label{item6}
    \end{enumerate}
\end{lemma}
\begin{proof}
    We have for all $q_1 \in (\declup, \infty)$ and $q_2 \in (1,\decllow)$ that
    \begin{equation}\label{eq:lambda_bnds}
        \lambda_j \gtrsim j^{-q_1} \quad \text{and} \quad \lambda_j \lesssim j^{-q_2}.
    \end{equation}
    Note that the set $\Mgeps$ is non-empty since $\varepsilon^2 < 1$.
    \begin{enumerate}
        \item Let $( \bsu \cup \{ k+1 \}, ( \bsj, j_{k+1}  ) ) \in M(\varepsilon, k+1)$ where $\emptyset \neq \bsu \subseteq [k], \bsj \in \N^{\bsu}$ and $j_{k+1} \in \N$. Since we have product weights, the definition \eqref{eq:opt_set_k} requires that $\gamma_{\bsu} \lambda_{\bsj} > \varepsilon^2 / \gamma_{k+1} \lambda_{j_{k+1}}$.
        Let $\ell = \max(\bsu)$, then $\ell \in \bsu \subseteq [\ell]$ and
        $$ (\bsu,\bsj) \in M\left( \frac{\varepsilon}{\sqrt{\gamma_{k+1} \lambda_{j_{k+1}}}}, \ell \right). $$

        \item Note that the right-hand side in statement \ref{item1} is a union of disjoint sets.

        \item Using the lower bound from \eqref{eq:lambda_bnds}, and the definition of $\Mbar(\varepsilon,k)$, we get for some $c_1 > 0$ that
        $$ c_1 \gamma_k (\overline{n}_k+1)^{-q_1} \leq \lambda_{\overline{n}_k+1} \gamma_k \leq \varepsilon^2 \implies \overline{n}_k \geq c_1^{1/q_1}\gamma_k^{1/q_1} \varepsilon^{-2/q_1}-1. $$

        \item Using the upper bound from \eqref{eq:lambda_bnds}, and the definition of $\Mbar(\varepsilon,k)$, we get for some  $c_2 > 0$ that
        $$ \varepsilon^2 < \gamma_k \lambda_{\overline{n}_k} \leq c_2 \gamma_k \overline{n}_k^{-q_2} \implies \overline{n}_k \leq c_2^{1/q_2} \gamma_k^{1/q_2} \varepsilon^{-2/q_2}. $$

        \item We prove this result by induction. For $k=1$, we have due to statement \ref{item4}
        \begin{equation*}%\label{eq:M1}
            n_1 = \abs{M(\varepsilon,1)} = \abs{\Mbar(\varepsilon,1)} \lesssim \gamma_1^{1/q_2} \varepsilon^{-2/q_2}.
        \end{equation*}
        We suppose now that the bound is true for $M(\varepsilon,1),\dots,M(\varepsilon,k)$. Using statements \ref{item2} and \ref{item4}, we get for some $c_2 > 0$
        \begin{equation*}
            \begin{aligned}
                \abs{M(\varepsilon,k+1)} &= \abs{\Mbar(\varepsilon,k+1)} + \sum_{\ell = 1}^{k} \sum_{j \in \N} \abs{M\left(\frac{\varepsilon}{\sqrt{\gamma_{k+1} \lambda_j}},\ell\right)} \\
                &\leq c_2 \varepsilon^{-2/q_2} \gamma_{k+1}^{1/q_2} + c_2 \sum_{\ell=1}^k \sum_{j \in \N} \left( \frac{\varepsilon}{\sqrt{\gamma_{k+1}\lambda_j}} \right)^{-2/q_2} \sum_{\ell \in \bsu \subseteq [\ell]} \gmu^{1/q_2} c(\bslambda,q_2)^{\abs{u}-1} \\
                &= c_2 \varepsilon^{-2/q_2} \gamma_{k+1}^{1/q_2} \left[ 1 + \sum_{\ell=1}^k \sum_{\ell \in \bsu \subseteq [\ell]} \gmu^{1/q_2} c(\bslambda,q_2)^{\abs{u}-1} \sum_{j \in \N} \lambda_j^{1/q_2} \right] \\
                &= c_2 \varepsilon^{-2/q_2} \left[ \gamma_{k+1}^{1/q_2} + \gamma_{k+1}^{1/q_2} \sum_{\emptyset \neq \bsu \subseteq [k]} \gmu^{1/q_2} c(\bslambda,q_2)^{\abs{u}} \right] \\
                &= c_2 \varepsilon^{-2/q_2} \sum_{k+1 \in \bsu \subseteq [k+1]} \gmu^{1/q_2} c(\bslambda,q_2)^{\abs{u}-1}.
            \end{aligned}
        \end{equation*}
        \item The estimate follows directly by applying \eqref{eq:prodWt} to the formula in statement \ref{item5}.
    \end{enumerate}
    This completes the proof.
\end{proof}

\begin{theorem}\label{th:ANOVA_rate}
    Let $\decglow,\decllow > 1$ and $\$(k) \asymp k^s$ for some $s>0$. Then the convergence rate of the $n$-th minimal errors using the nested subspace sampling cost model in ANOVA function spaces with product weights is bounded as follows
    $$ \min\left\{ \frac{\decllow}{2}, \frac{\decglow}{2(1+s)} \right\} \leq \ConvAllInf^A \leq \min\left\{ \frac{\decllow}{2}, \frac{\decgup}{2(1+s)} \right\}. $$
\end{theorem}
\begin{proof}
    Let $\varepsilon$ be the desired worst case error, and $\meps$ be the highest level for our algorithm.

    We first show $ \ConvAllInf^A \leq \min\left\{ \frac{\decllow}{2}, \frac{\decgup}{2(1+s)} \right\} $ using $\$(k) \gtrsim k^s$. The inequality $\ConvAllInf \leq \decllow/2$ holds since the infinite-variate problem is at least as difficult as the univariate problem. It only remains to consider the case when $\decgup/(1+s) < \decllow \leq \declup$. Since $\Mbar(\varepsilon,k) \subset M(\varepsilon,k)$, we use the lower bound on $\abs{\Mbar(\varepsilon,k)}$ from statement \ref{item3} of \RefLem{lem:M_eps_k}, and \eqref{eq:gamma_bnds} to get
    \begin{equation}\label{eq:lCostA_nss}
        \cost(\AOpt) \geq \sum_{k=1}^{\meps} \abs{\Mbar(\varepsilon,k)} \cdot \$(k) \gtrsim  \varepsilon^{-2/q_1} \sum_{k=1}^{\meps} \gamma_k^{1/q_1} k^s \gtrsim \varepsilon^{-2/q_1} \sum_{k=1}^{\meps} k^{s-p_1/q_1} ,
    \end{equation}
    where for $\delta > 0$ arbitrarily small, we chose $p_1 = \decgup(1+\delta)$ and $q_1 = \declup(1 + \delta)$. Consequently, we have $s - p_1/q_1 > -1$, and we use Lemma \ref{lem:m_eps_UL} to estimate the sum in \eqref{eq:lCostA_nss} as
    $$ \sum_{k=1}^{\meps} k^{s-p_1/q_1} \gtrsim (\meps)^{1-p_1/q_1+s} \gtrsim \varepsilon^{-2(1-p_1/q_1+s)/p_1}, $$
    which implies that $\cost(\AOpt) \gtrsim \varepsilon^{-2(1+s)/p_1}$. As $\delta \to 0$, we have $p_1 \to \decgup$ and $\ConvAllInf^A \leq \decgup/2(1+s)$.
    
    We now show that $ \ConvAllInf^A \geq \min\left\{ \frac{\decllow}{2}, \frac{\decglow}{2(1+s)} \right\}$ using $\$(k) \lesssim k^s$. We use \eqref{eq:gamma_bnds} and statement \ref{item6} from Lemma \ref{lem:M_eps_k} to get
    \begin{equation}\label{eq:uCostA_nss}
        \begin{aligned}
            \cost(\AOpt) = \sum_{k=1}^{\meps} \abs{M(\varepsilon,k)} \cdot \$(k) \lesssim \varepsilon^{-2/q_2} \sum_{k=1}^{\meps} \gamma_k^{1/q_2} k^s \lesssim \varepsilon^{-2/q_2} \sum_{k=1}^{\meps} k^{s-p_2/q_2} ,
        \end{aligned}
    \end{equation}
    where we choose $\delta > 0$ arbitrarily small so that $p_2 = \decglow(1 - \delta)$ and $q_2 = \min\{\decglow,\decllow\}(1 - \delta) > 1$. We analyze the last expression in \eqref{eq:uCostA_nss} in three cases.

    \begin{enumerate}
        \item $ \decglow/(1+s) > \decllow \implies s - p_2/q_2 < -1 $: Then $ \sum_{k=1}^{\meps} k^{s-p_2/q_2} \lesssim 1 $. Since $\decg > \decllow$ in this case, we have $q_2 = \decllow(1-\delta)$. Therefore, as $\delta \to 0$, we get $q_2 \to \decllow$ and $\ConvAllInf^A \geq \decllow/2$.
        
        \item $ \decglow/(1+s) = \decllow \implies s - p_2/q_2 = -1 $: For this case, we get the estimate 
        $\sum_{k=1}^{\meps} k^{-1} \lesssim 1 + \ln(\meps) \lesssim 1+\frac{2}{p_2} \ln(1/\varepsilon)$, where we used Lemma \ref{lem:m_eps_UL} for the last relation. We set $q_2 = \decllow(1-\delta)$ since $\decglow > \decllow$. Consequently, we have $\cost(\AOpt) \lesssim \varepsilon^{-2/q_2}(1+\ln(1/\varepsilon))$. We can ignore the logarithmic term as $\delta \to 0$ to get $q_2 \to \decllow$ and $\ConvAllInf^A \geq \decllow/2$.
        
        \item $\decglow/(1+s) < \decllow \implies s - p_2/q_2 > -1$: For an appropriate choice of $q_2$, and due to Lemma \ref{lem:m_eps_UL} we get
        $$ \sum_{k=1}^{\meps} k^{-p_2/q_2+s} \lesssim \meps^{1-p_2/q_2+s} \lesssim \varepsilon^{-2(1-p_2/q_2+s)/p_2}. $$
        Consequently, $\cost(\AOpt) \lesssim \varepsilon^{-2(1+s)/p_2}$ which is independent of $q_2$. Hence, as $\delta \to 0$, we set $p_2 \to \decglow$ to get $\ConvAllInf^A \geq \decglow/2(1+s)$.
    \end{enumerate}
    This completes the proof.
\end{proof}

\section{Non-ANOVA spaces}

Let us now turn to the case where the underlying function decomposition of our Hilbert space $\Hg$ is not of ANOVA-type. In this section, we assume that $1 < \decglow \leq \decgup < \infty$. We start by proving an upper bound on the convergence rate of $n$-th minimal errors.

\begin{proposition}\label{pr:nA_low_err}
   Let $\$(k) \gtrsim k^s$ for any $k \in \N$ and some $s > 0$. Then the optimal convergence rate for the algorithms defined as in \eqref{eq:lin_algo_all} satisfy in non-ANOVA function spaces
   \begin{equation*}
      \ConvAllInf \leq \min \left\{ \frac{\decllow}{2}, \frac{\decgup - 1}{2s} \right\} .
   \end{equation*}
\end{proposition}
\begin{proof}
   The bound $\ConvAllInf \leq \decllow/2$ follows from the rate of convergence for the univariate problem. % This can be shown by considering for each $g \in H(1+k)$ the function $f \in \Hg$ given by $f(\bsx) := g(x_1)$ for all $\bsx \in \X$.
   
   To show the other part, we adapt the proof strategy from \cite[Proposition~4]{Was12}. Let $\A_N$ be a linear algorithm using finitely many function evaluations, such that $\cost(\A_N) \leq N$. We now choose $L := \sup \left\{ k : \$(k) \leq N \right\}.$
   Due to this definition, $[L]$ is a superset of all the indices that the algorithm $\A_N$ relies on. Since we assumed that $\$(k) \gtrsim k^s$, there exists a constant $c>0$ such that $k^s/c^s \leq \$(k)$, implying $L \leq c N^{1/s}$.
   
   Consider a univariate function $h \in H(k)$ such that $\norm{h}_{H(k)} = 1$ and $c_1 := \int_D h(x) \drho(x) \neq 0$, cf. Remark \ref{re:ANOVA}. Then define $f^\ast \in \Hg$ as
   \begin{equation*}
       f^\ast(\bsx) := \frac{1}{\kappa} \sum_{j > L} \gamma_j h(x_j) , \quad \text{with } \kappa := \sqrt{ \sum_{j > L} \gamma_j } .
   \end{equation*}
   Accordingly, we get $\A_N(f^\ast) = 0$ due to our choice of $L$ and $f^\ast$, as shown in \cite{Was12}. Moreover, $\norm{f^\ast}_{\Hg} = 1$, and also
   $$ \norm{f^\ast}_{\Ltwo}^2 = \left( \norm{h}_{L^2(D,\rho)}^2 - c_1^2 \right) \frac{\sum_{j > L} \gamma_j^2}{\sum_{j > L} \gamma_j} + c_1^2 \sum_{j > L} \gamma_j. $$

   We need an estimate only for the second term in the equation above since the first term is positive but much smaller. It follows from \RefLem{le:decay_bounds} that $\sum_{j > L} \gamma_j \gtrsim L^{1-\decgup - \delta} \gtrsim N^{(1-\decgup - \delta)/s}$ for all $\delta > 0$. From the construction of $f^\ast$, we have for the error that
   \begin{equation*}
       e(\A_N; \Hg) = \sup_{\norm{f}_{\Hg} = 1} \norm{f - \A_N(f)}_{\Ltwo} \geq \norm{f^\ast}_{\Ltwo} \gtrsim N^{(1-\decgup - \delta)/2s}.
   \end{equation*}
   This shows that $\ConvAllInf \leq (\decgup - 1)/2s$, and the result follows.
\end{proof}

For non-ANOVA spaces, Lemma \ref{lemma:decompW} does not hold in general, which means that we do not know how to express the eigenpairs of $\W : \Hg \to \Hg$ in terms of the eigenpairs of $\Wu$.
% Consequently, in comparison to the ANOVA setting, \eqref{eq:ortho_Hg} holds also in this setting but \eqref{eq:ortho_L2} does not.
For the upper error bound, we use results from \cite{Was12} which are stated below for the reader's convenience.

\begin{lemma}\label{le:scaled_Wt_nA}
    For all $c \in (1/\decglow , 1)$ it holds that
    % \begin{equation}
        $\sup_{\bsu \in \UUg} \gmu^{1-c} \norm{S}^{2\abs{\bsu}} < \infty.$
    % \end{equation}
    \begin{enumerate}
        \item For
        $$ \ghc = \{ \widehat{\gamma}_{\bsu,c} \}_{\bsu \in \UUg}, \quad \gmh_{\bsu,c} := \gmu^{1-c},$$
        we define $ \phi : \Hg \to \Hghc, f \mapsto \sum_{\bsu \in \UUg} \gamma_{\bsu}^{-c/2} \fu $. Then $\phi$ is an isometric isomorphism between $\Hg$ and $\Hghc$. Furthermore, $\Hg \subseteq \Hghc$.

        \item Let $\A$ be a linear continuous algorithm satisfying
        \begin{equation}\label{eq:algo_cond}
            \A(\Hu) \subseteq L^2(\Du, \rhou),
        \end{equation}
        for all $\bsu \in \UUg$. Then we get the upper bound
        \begin{equation}\label{eq:nA_err}
            \norm{\Sol(f) - \A(f)}_{\Ltwo}^2 \leq C_{\bsgamma} \sum_{\bsu \in \UUg}  \norm{S_{\bsu}(\fh_{\bsu}) - A_{\bsu} (\fh_{\bsu})}^2_{\Ltwou}
        \end{equation}
        where $C_{\bsgamma} = \sum_{\bsv \in \UUg} \gmv^c$, $\phi(f) =: \fh = \sum_{\bsu \in \UUg} \fh_\bsu$ for all $f \in \Hg$, and $\Au : \Hu \to \Ltwou$ such that $\Au(f) = \A(f)$ for all $f \in \Hu$.
        %$$ e(\A, \Sol, \Hg)^2 \leq \sum_{\bsu \in \UUg}  e(\Au, \Su, \tilde{\Hu})^2, $$ where $\tilde{\Hu} = \frac{C_{\bsgamma}}{\gamma_{\bsu}^c} \Hu$.
    \end{enumerate}
\end{lemma}
It is straightforward to check that $\phi$ defines an isometric isomorphism. $\Hg \subseteq \Hghc$ is due to \cite[Lemma~2]{Was12}. The inequality \eqref{eq:nA_err} has been derived in the proof of \cite[Lemma~2]{Was12} for any algorithm $\A$ satisfying \eqref{eq:algo_cond}.

\begin{proposition}\label{Prop:low_bound_non-ANOVA}
    For any $\varepsilon \in (0,1)$, define the set $\Mgheps$ as in \eqref{eq:opt_set} and also the algorithm $\AOptc : \Hghc \to \Ltwo$ as defined in \eqref{eq:opt_algo}. Then the algorithm $\AOptc$ satisfies \eqref{eq:algo_cond} and $e(\AOptc, \Sol, \Hg) \leq \varepsilon$. The algorithm $\AOptc$ establishes that
    \begin{equation*}
        \ConvAllInf \geq \min \left\{ \frac{\decllow}{2}, \frac{\decglow-1}{2(1+s)} \right\}.
    \end{equation*}
\end{proposition}
\begin{proof}
    Using Lemmas \ref{le:scaled_Wt_nA} and \ref{le:opt_err_u}, we get
    \begin{equation*}
        \begin{aligned}
            &\norm{\Sol(f) - \AOptc(f)}_{\Ltwo}^2 \lesssim \sum_{\bsu \in \UUg} \norm{S_{\bsu}(\fh_{\bsu}) - A_{\bsu,\varepsilon/\sqrt{\gmh_{\bsu,c}}}^\ast (\fh_{\bsu})}^2_{L^2(D^{\bsu},\rho^{\bsu})} \\
            &\leq \varepsilon^2 \sum_{\bsu \in \UUg} \frac{1}{\gmh_{\bsu,c}} \norm{\fh_\bsu}^2_{\Hu} = \varepsilon^2 \norm{\fh}^2_{\Hghc} = \varepsilon^2 \norm{f}^2_{\Hg}.
        \end{aligned}
    \end{equation*}

    Notice that $\AOptc$ has the same cost in $\Hghc$ and in its subspace $\Hg$, since the function space decomposition in both spaces rely on the same family of subspaces $(\Hu)_{\bsu \in \UUg}$. Let $\declow_{\ghc}$ denote the lower decay rate of product weights $\ghc$. Then from \eqref{eq:uCostA_nss}, and the subsequent analysis in Theorem \ref{th:ANOVA_rate}, we get
    $$ \ConvAllInf \geq \min\left\{ \frac{\decllow}{2} , \frac{\declow_{\ghc}}{2(1+s)} \right\}. $$
    Finally, $\declow_{\ghc} = \decglow(1-c) \xrightarrow{c \to 1/\decglow} \decglow-1$. This completes the proof.
\end{proof}

\begin{corollary}\label{cor:comparison}
    Assume that $\decglow = \decgup =: \decg$. If $1 < \decg < 1+s$, then the convergence rate of $L^2$-approximation in ANOVA spaces is strictly larger than the one in non-ANOVA spaces.
\end{corollary}
\begin{proof}
Note that due to $1 < \decllow$, the condition $\decg < 1+s$ implies $\decllow > (\decg - 1)/s$.
Hence, under the given condition, the exact convergence rate of $n$-th minimal errors in the ANOVA case from Theorem \ref{th:ANOVA_rate} is greater than the largest possible convergence rate  in the non-ANOVA case from Proposition \ref{pr:nA_low_err}, i.e.
$$ \min \left\{\frac{\decllow}{2}, \frac{\decg - 1}{2s}\right\} = \frac{\decg - 1}{2s} < \min \left\{\frac{\decllow}{2}, \frac{\decg}{2(1+s)}\right\} = \frac{\decg}{2(1+s)}. $$
The statement is hence proved.
\end{proof}

As we see from the analysis, there is a gap between the convergence rates in the ANOVA and the non-ANOVA setting. Note that we can use the function $f^\ast$ from \RefProp{pr:nA_low_err} to obtain the same upper bound on the convergence rate for the integration problem in the non-ANOVA setting, $I : \Hg \to \R$. More specifically, we have $I(f^\ast) = c_1 \sqrt{\sum_{j > L} \gamma_j } \gtrsim N^{(1 - \decgup - \delta)/2s}$, which implies that the convergence rate is at most $(\decgup - 1)/2s$. On the other hand, the integration problem is trivial in the ANOVA setting. This provides some intuition for the gap since we know that the function approximation problem is, in general, not easier than the integration problem.

\begin{acknowledgement}
    We thank an anonymous referee for helpful comments, and for pointing out the second example in Remark \ref{re:decay_ex}.
\end{acknowledgement}

\section*{Appendix}
\addcontentsline{toc}{section}{Appendix}

For a null sequence of positive reals $\bsx := \{ x_j \}_{j \in \N}$, set $\declow_1 := \declow_{\bsx}$ as in \eqref{def:decay} and define the following quantities:

\begin{eqnarray}
    \declow_{2} &:= &\sup \{ \alpha \geq 0 : \lim_{n \to \infty} x_n n^\alpha = 0 \}, \quad \text{and} \label{def:dec_low2}\\
    \declow_3 &:= &\sup \{ \alpha \geq 0 : \sup_{n \in \N} x_n n^\alpha < \infty \}. \label{def:dec_low3}
\end{eqnarray}

\begin{lemma}\label{le:dec_low_equiv}
    If $\bsx$ is a non-increasing null sequence of positive reals, then we have the following identity
    $$ \declow_1 = \declow_2 = \declow_3. $$
\end{lemma}
\begin{proof}
    We first claim $\declow_1 \leq \declow_3$. Let $\alpha < \declow_1$. Using the monotonicity of $\bsx$, we get for all $N \in \N$
    $$ Nx_N^{1/\alpha} \leq \sum_{n=1}^N x_n^{1/\alpha} \leq \sum_{n=1}^\infty x_n^{1/\alpha} =: c_1 < \infty.$$ % x_N \leq c^\alpha N^{-\alpha}.$$
    Therefore, for all $N \in \N$, we have $x_N N^{\alpha} \leq c_1^{\alpha}$, implying $\alpha \leq \declow_3$. As $\alpha \to \declow_1$, we have the claim.

    Next, we show $\declow_3 \leq \declow_2$. Let $\alpha < \declow_3$. Then we know that $\{ x_n n^\alpha \}$ is bounded. Consequently, for every $\varepsilon > 0$, the sequence $\{ x_n n^{\alpha - \varepsilon} \}$ converges to zero, and $\alpha - \varepsilon \leq \declow_2$. In the limit as $\varepsilon \to 0$ and $\alpha \to \declow_3$, we have the result.

    Finally, we prove $\declow_2 \leq \declow_1$. Let $\alpha < \declow_2$ and $c_2 > 0$. From \eqref{def:dec_low2}, we know that there exists $N_2 \in \N$ such that for all $n \geq N_2$, we have $x_n \leq c_2 n^{-\alpha}$. Then for all $\varepsilon > 0$
    \begin{equation*}
        \sum_{n=1}^\infty x_n^{1/(\alpha - \varepsilon)} = \sum_{n=1}^{N_2-1} x_n^{1/(\alpha - \varepsilon)} + \sum_{n={N_2}}^{\infty} x_n^{1/(\alpha - \varepsilon)} \leq \sum_{n=1}^{N_2-1} x_n^{1/(\alpha - \varepsilon)} + c_2^{\frac{1}{\alpha - \varepsilon}}\sum_{n={N_2}}^{\infty} \frac{1}{n^{\frac{\alpha}{\alpha-\varepsilon}}} < \infty,
    \end{equation*}
    which means $\alpha - \varepsilon \leq \declow_1$. The result follows as $\varepsilon \to 0$ and $\alpha \to \declow_2$.

    The identity is hence proved.
\end{proof}

\begin{lemma}\label{le:dec_up_low}
    Let $\decup_{\bsx}$ be defined as in \eqref{def:decay_up}. If $\bsx$ is a non-increasing null sequence of positive reals, then we have the following inequality
    \begin{equation*}
        \declow_\bsx \leq \decup_{\bsx}.
    \end{equation*}
\end{lemma}
\begin{proof}
    Due to \RefLem{le:dec_low_equiv}, we know $\declow_\bsx = \declow_3$. Let $\alpha < \declow_3$. Then for all $\varepsilon > 0$
    \begin{equation*}
        \sup_{n \in \N} x_n n^\alpha < \infty \implies \inf_{n \in \N} x_n n^{\alpha - \varepsilon} = 0 \implies \alpha-\varepsilon \leq \decup_{\bsx}.
    \end{equation*}
    Taking the limit as $\varepsilon \to 0$ and $\alpha \to \declow_3$, we get the desired inequality.
\end{proof}

\begin{remark}\label{re:decay_ex}
    We provide two examples of sequences to show that the lower and upper decay rates are equal in most relevant cases but can also differ significantly. As first example, let $\{x_n\}_{n \in \N}$ behave as follows
    \begin{equation*}
        c_1 n^{-\alpha} \ln(n+1)^{\beta_1} \leq x_n \leq c_2 n^{-\alpha} \ln(n+1)^{\beta_2},
    \end{equation*}
    for $n$ large enough, where $\alpha > 0, c_1,c_2 > 0, \beta_1,\beta_2 \in \R$ and $\beta_1 \leq \beta_2$. Then, we show $\declow_{\bsx} = \decup_{\bsx} = \alpha$. For all $\varepsilon > 0$, we have
    \begin{equation*}
        \lim_{n \to \infty} x_n n^{\alpha-\varepsilon} \leq c_2 \lim_{n \to \infty} \frac{\ln(n+1)^{\beta_2}}{n^\varepsilon} = 0.
    \end{equation*}
    Since we consider only positive values for $x_n$, we actually get $\lim_{n \to \infty} x_n n^{\alpha-\varepsilon} = 0$. Hence, using \eqref{def:dec_low2}, we get $\declow_{\bsx} \geq \alpha$. Moreover, for all $\varepsilon > 0$, we also get
    \begin{equation*}
        \inf_{n \in \N} x_n n^{\alpha + \varepsilon} \geq c_1 \inf_{n \in \N} \ln(n+1)^{\beta_1} n^\varepsilon > 0.
    \end{equation*}
    Therefore, using \eqref{def:decay_up}, we have $\decup_\bsx \leq \alpha$. From \RefLem{le:dec_up_low} follows $\declow_{\bsx} = \alpha = \decup_{\bsx}$.
    
    In contrast, define recursively $n_1 = 0$ and $n_{k+1} = 2^{n_k}+1$ and consider the sequence $y_n = 2^{-n_k}$ for $n \in [n_k,n_{k+1})$. In this case, we show $\declow_\bsy = 1$ and $\decup_\bsy = \infty$. Let $0 < \alpha < 1$. Then
    \begin{equation*}
        \sum_{n \in \N} y_n^{1/\alpha} = \sum_{k \in \N} \sum_{n = n_k}^{n_{k+1}-1} 2^{-n_k/\alpha} = \sum_{k \in \N} \frac{2^{n_k} - n_k}{2^{n_k/\alpha}} \leq \sum_{k \in \N} 2^{-n_k(\frac{1}{\alpha}-1)} < \infty,
    \end{equation*}
    which implies $\alpha \leq \declow_\bsy$ and hence $\declow_\bsy \geq 1$. For $\alpha = 1$, the series above diverges.
    Therefore, we have $\declow_\bsy \leq 1$, and altogether $\declow_\bsy = 1$. On the other hand, we have for all $\alpha > 0$
    \begin{equation*}
        \inf_{n \in \N} y_n n^\alpha = \inf_{k \in \N} \min_{n \in [n_k, n_{k+1})} y_n n^\alpha = \inf_{k \in \N} \min_{n \in [n_k, n_{k+1})} \frac{n^\alpha}{2^{n_k}} = \inf_{n \in \N} \frac{n^\alpha}{2^{n}} = 0.
    \end{equation*}
    Consequently, $\decup_\bsy = \infty$.
\end{remark}

\begin{lemma}\label{le:decay_bounds}
    Let $\bsx$ be a non-increasing null sequence of positive reals with lower and upper decay rates denoted by $\declow_\bsx$ and $\decup_\bsx$. Then, for all $p_1 \in (\decup_\bsx,\infty)$ and $p_2 \in (0, \declow_\bsx)$
    \begin{equation*}
        n^{-p_1} \lesssim x_n \lesssim n^{-p_2}.
    \end{equation*}
\end{lemma}
\begin{proof}
    From the definition in \eqref{def:decay_up}, we have $c_1 := \inf_{n \in \N} x_n n^{p_1}> 0$. Consequently, we have $x_n n^{p_1} \geq c_1$ and hence $x_n \geq c_1 n^{-p_1}$.

    We have shown in \RefLem{le:dec_low_equiv} that $\declow_\bsx = \declow_3$. Accordingly, from the definition in \eqref{def:dec_low3}, we have $c_2 := \sup_{n \in \N} x_n n^{p_2} < \infty$. This implies that for all $n \in \N$, $x_n n^{p_2} \leq c_2$ and hence $x_n \leq c_2 n^{-p_2}$.
\end{proof}

\bibliographystyle{is-abbrv}
\bibliography{references}

\begin{thebibliography}{10}
\ifx \showCODEN  \undefined \def \showCODEN #1{CODEN #1}  \fi
\ifx \showISBN   \undefined \def \showISBN  #1{ISBN #1}   \fi
\ifx \showISSN   \undefined \def \showISSN  #1{ISSN #1}   \fi
\ifx \showLCCN   \undefined \def \showLCCN  #1{LCCN #1}   \fi
\ifx \showPRICE  \undefined \def \showPRICE #1{#1}        \fi
\ifx \showURL    \undefined \def \showURL {URL }          \fi
\ifx \path       \undefined \input path.sty               \fi
\ifx \ifshowURL \undefined
     \newif \ifshowURL
     \showURLtrue
\fi

\bibitem{Aro50}
N.~Aronszajn.
\newblock Theory of reproducing kernels.
\newblock {\em Trans. Amer. Math. Soc.}, 68:\penalty0 337--404, 1950.

\bibitem{BG14}
J.~Baldeaux and M.~Gnewuch.
\newblock Optimal randomized multilevel algorithms for infinite-dimensional
  integration on function spaces with {ANOVA}-type decomposition.
\newblock {\em SIAM J. Numer. Anal.}, 52:\penalty0 1128--1155, 2014.

\bibitem{DG14a}
J.~Dick and M.~Gnewuch.
\newblock Infinite-dimensional integration in weighted {H}ilbert spaces:
  anchored decompositions, optimal deterministic algorithms, and higher order
  convergence.
\newblock {\em Found. Comput. Math.}, 14:\penalty0 1027--1077, 2014.

\bibitem{DG14b}
J.~Dick and M.~Gnewuch.
\newblock Optimal randomized changing dimension algorithms for
  infinite-dimensional integration on function spaces with {ANOVA}-type
  decomposition.
\newblock {\em J. Approx. Theory}, 184:\penalty0 111--145, 2014.

\bibitem{FHW12}
G.~E. Fasshauer, F.~J. Hickernell, and H.~Wo\'{z}niakowski.
\newblock On dimension-independent rates of convergence for function
  approximation with {G}aussian kernels.
\newblock {\em SIAM Journal on Numerical Analysis}, 50\penalty0 (1):\penalty0
  247--271, 2012.

\bibitem{GHHR17}
M.~Gnewuch, M.~Hefter, A.~Hinrichs, and K.~Ritter.
\newblock Embeddings of weighted {H}ilbert spaces and applications to
  multivariate and infinite-dimensional integration.
\newblock {\em Journal of Approximation Theory}, 222:\penalty0 8--39, 2017.

\bibitem{GHHRW19}
M.~Gnewuch, M.~Hefter, A.~Hinrichs, K.~Ritter, and G.~W. Wasilkowski.
\newblock Embeddings for infinite-dimensional integration and
  ${L}_2$-approximation with increasing smoothness.
\newblock {\em J.\ Complexity}, 54:\penalty0 101406, 2019.

\bibitem{GHRR23}
M.~Gnewuch, A.~Hinrichs, K.~Ritter, and R.~R{\"u}{\ss}mann.
\newblock Infinite-dimensional integration and ${L}^2$-approximation on
  {H}ermite spaces.

\bibitem{GMR14}
M.~Gnewuch, S.~Mayer, and K.~Ritter.
\newblock On weighted {H}ilbert spaces and integration of functions of
  infinitely many variables.
\newblock {\em J. Complexity}, 30:\penalty0 29--47, 2014.

\bibitem{HR13}
M.~Hefter and K.~Ritter.
\newblock On embeddings of weighted tensor product {H}ilbert spaces.
\newblock {\em J. Complexity}, 31:\penalty0 405--423, 2015.

\bibitem{HMNR10}
F.~J. Hickernell, T.~M{\"u}ller-Gronbach, B.~Niu, and K.~Ritter.
\newblock Multi-level {M}onte {C}arlo algorithms for infinite-dimensional
  integration on $\mathbb{R}^{\N}$.
\newblock {\em J. Complexity}, 26:\penalty0 229--254, 2010.

\bibitem{IKP16}
C.~Irrgeher, P.~Kritzer, F.~Pillichshammer, and H.~Wo\'zniakowski.
\newblock Approximation in {H}ermite spaces of smooth functions.
\newblock {\em Journal of Approximation Theory}, 207:\penalty0 98--126, 2016.
\newblock \showISSN{0021-9045}.

\bibitem{KSWW10a}
F.~Y. Kuo, I.~H. Sloan, G.~W. Wasilkowski, and H.~Wo\'zniakowski.
\newblock Liberating the dimension.
\newblock {\em J. Complexity}, 26:\penalty0 422--454, 2010.

\bibitem{MR09}
T.~M{\"u}ller-Gronbach and K.~Ritter.
\newblock Variable subspace sampling and multi-level algorithms.
\newblock In P.~L'Ecuyer and A.~B. Owen, editors, {\em Monte Carlo and
  Quasi-Monte Carlo Methods 2008}. Springer, 2009.

\bibitem{NW08}
E.~Novak and H.~Wo\'zniakowski.
\newblock {\em Tractability of {M}ultivariate {P}roblems. Vol. 1: {L}inear
  {I}nformation}.
\newblock EMS Tracts in Mathematics. European Mathematical Society (EMS),
  Z\"urich, 2008.

\bibitem{TWW88}
J.~F. Traub, G.~W. Wasilkowski, and H.~Wo\'zniakowski.
\newblock {\em Information-Based Complexity}.
\newblock Academic Press, New York, 1988.

\bibitem{Was12}
G.~W. Wasilkowski.
\newblock Liberating the dimension for ${L}_2$-approximation.
\newblock {\em J. Complexity}, 28:\penalty0 304--319, 2012.

\bibitem{WW11a}
G.~W. Wasilkowski and H.~Wo\'zniakowski.
\newblock Liberating the dimension for function approximation.
\newblock {\em J. Complexity}, 27:\penalty0 86--110, 2011.

\bibitem{Wnuk22}
M.~Wnuk.
\newblock A short note on compact embeddings of reproducing kernel {H}ilbert
  spaces in ${L}^2$ for infinite-variate function approximation.
\newblock In A.~Hinrichs, P.~Kritzer, and F.~Pillichshammer, editors, {\em
  Monte Carlo and Quasi-Monte Carlo Methods 2022}. Springer Verlag, 2022.

\end{thebibliography}
\end{document}